\documentclass[%
  a4paper,
  fontsize=11pt,
  oneside,
  numbers=enddot,
]{scrartcl}
\KOMAoptions{DIV=12}					
\usepackage[obeyspaces,hyphens,spaces]{url}
\usepackage[fleqn]{amsmath}
\usepackage{amsthm}
\allowdisplaybreaks[1]              

\usepackage[T1]{fontenc}
\usepackage{textcomp}
\usepackage{amsfonts}
\usepackage{amssymb}
\usepackage[proportional]{libertine}
\usepackage[%
  libertine,
  slantedGreek,
  nosymbolsc,
  nonewtxmathopt,
  subscriptcorrection,
]{newtxmath}
\usepackage[scaled=0.95,varqu,varl]{inconsolata}

\usepackage[scr=boondoxo]{mathalpha}
\usepackage{euscript}               

\usepackage[svgnames]{xcolor}
\definecolor{structure}{HTML}{0455BF}
\definecolor{dblue}{HTML}{0455BF}
\definecolor{dgreen}{HTML}{02724A}
\definecolor{dred}{HTML}{D90404}
\definecolor{dviolet}{HTML}{42208C}
\definecolor{labelkey}{HTML}{0455BF}
\definecolor{refkey}{HTML}{0455BF}

\addtokomafont{section}{\color{structure}\centering}
\addtokomafont{subsection}{\color{structure}}
\addtokomafont{title}{\color{structure}}

\usepackage[headsepline]{scrlayer-scrpage}
\addtokomafont{headsepline}{\color{structure}}
\addtokomafont{pageheadfoot}{\color{structure}}
\addtokomafont{pagenumber}{\color{structure}}
\makeatletter
  \ihead{\textup{arXiv: 2501.11662v3}}
  \ohead{\pagemark}
\makeatother

\usepackage{enumitem}

\setlist{itemsep=-1.5pt}

\usepackage{upref}  
\usepackage{hyperref}
\hypersetup{%
  colorlinks=true,
  linktocpage=true,
  linkcolor=dgreen,
  citecolor=dgreen,
  urlcolor=dred,
  pdfencoding=auto,
  hypertexnames=false,
}

\usepackage[capitalize,nosort,nameinlink]{cleveref}
\crefdefaultlabelformat{#2\textup{#1}#3}
\crefrangeformat{enumi}{#3\textup{#1}#4\textup{\textendash}#5\textup{#2}#6}
\crefname{theorem}{Theorem}{Theorems}
\crefname{proposition}{Proposition}{Propositions}
\crefname{corollary}{Corollary}{Corollaries}
\crefname{lemma}{Lemma}{Lemmas}
\crefname{fact}{Fact}{Facts}
\crefname{definition}{Definition}{Definitions}
\crefname{example}{Example}{Examples}
\crefname{assumption}{Assumption}{Assumptions}
\crefname{remark}{Remark}{Remarks}
\crefname{notation}{Notation}{Notations}
\crefname{item}{}{}
\crefname{enumi}{}{}
\crefname{equation}{}{}

\makeatletter
  \g@addto@macro\th@plain{%
    \thm@headfont{\color{structure}\scshape\bfseries}
    \thm@notefont{\normalfont}
  }
  \g@addto@macro\th@definition{%
    \thm@headfont{\color{structure}\scshape\bfseries}
    \thm@notefont{\normalfont}
  }
  \g@addto@macro\th@remark{
    \thm@headfont{\color{structure}\scshape\bfseries}
    \thm@notefont{\normalfont}
  }
\makeatother

\theoremstyle{plain}
\newtheorem{theorem}{Theorem}

\newtheorem{corollary}[theorem]{Corollary}
\newtheorem{lemma}[theorem]{Lemma}

\theoremstyle{definition}

\newtheorem{notation}[theorem]{Notation}
\newtheorem{example}[theorem]{Example}

\theoremstyle{remark}
\newtheorem{remark}[theorem]{Remark}

\usepackage[extdef=true]{delimset}
\DeclareMathDelimiterSet{\scal}[2]{
  \selectdelim[l]< {#1} \mathpunct{}\selectdelim[p]| {#2} \selectdelim[r]>
}
\renewcommand{\pair}{\delimpair<{[.],}>}
\newcommand{\menge}[2]{
  \bigl\{{#1}\mid{#2}\bigr\}
  }
\DeclareMathDelimiterSet{\Menge}[2]{
  \selectdelim[l]\{{#1}\selectdelim[m]|{#2}\selectdelim[r]\}
  }
\newcommand*\Cdot{{\mkern 1.6mu\cdot\mkern 1.6mu}}

\newcommand{\RR}{\mathbb{R}}

\newcommand{\HH}{\mathcal{H}}
\newcommand{\XX}{\mathcal{X}}
\newcommand{\YY}{\mathcal{Y}}
\newcommand{\GG}{\mathcal{G}}
\newcommand{\HHH}{\boldsymbol{\mathcal{H}}}

\newcommand{\kt}{\boldsymbol{\EuScript{K}}}

\newcommand{\pinf}{{+}\infty}
\newcommand{\minf}{{-}\infty}
\newcommand{\emp}{\varnothing}
\newcommand{\exi}{\exists\,}

\renewcommand{\geq}{\geqslant}

\newcommand{\RX}{\intv[l]{\minf}{\pinf}}

\DeclareMathOperator{\ran}{ran}
\DeclareMathOperator{\cran}{\overline{ran}}

\DeclareMathOperator{\inte}{int}

\newcommand{\Id}{\mathrm{Id}}

\DeclareMathOperator{\conv}{conv}
\DeclareMathOperator{\cconv}{\overline{conv}}
\DeclareMathOperator{\caff}{\overline{aff}}
\DeclareMathOperator{\aff}{aff}
\DeclareMathOperator{\dom}{dom}

\DeclareMathOperator{\gra}{gra}
\DeclareMathOperator{\rint}{rint}

\renewcommand\abstractname{Abstract}

\newcommand\mscname{2020 Mathematics Subject Classification}
\renewenvironment{abstract}{%
  \vspace*{-0.50cm}
  \small
  \quotation%
  \noindent%
  {\normalfont\scshape\bfseries
    \nobreak\textcolor{structure}{\abstractname.}\ }%
  }{%
  \endquotation%
  }
\newenvironment{keywords}
{\renewcommand\abstractname{\keywordsname}\begin{abstract}}
{\end{abstract}}
\newenvironment{msc}
{\renewcommand\abstractname{\mscname}\begin{abstract}}
{\end{abstract}}

\usepackage[%
  sorting=nyvt,
  style=ext-numeric,
  giveninits=true,
  articlein = false,
]{biblatex}
\DeclareFieldFormat[article,inbook,inproceedings]{title}{#1\isdot}
\DeclareFieldFormat[article]{eid}{art.~#1}
\DeclareFieldFormat{edition}{#1~ed\isdot}
\renewbibmacro*{publisher+location+date}{%
  \printlist{publisher}%
  \iflistundef{location}
  {\setunit*{\addcomma\space}}
  {\setunit*{\addcomma\space}}%
  \printlist{location}%
  \setunit*{\addcomma\space}%
  \usebibmacro{date}%
  \newunit}
\AtEveryBibitem{%
  \clearfield{note}%
  }
\renewbibmacro*{maintitle+title}{%
  \iffieldsequal{maintitle}{title}
    {\clearfield{maintitle}%
     \clearfield{mainsubtitle}%
     \clearfield{maintitleaddon}}
    {\iffieldundef{maintitle}
       {}
       {\usebibmacro{maintitle}%
    \newunit\newblock
    \iffieldundef{volume}
      {}
      {\printfield{volume}%
           \printfield{part}%
           \setunit{\addcolon\space}}}}%
  \usebibmacro{title}%
  \iffieldundef{edition}
    {}%
    {\addcomma}%
  \newunit%
  }

\usepackage[auth-sc]{authblk}
\newcommand{\email}[1]{\href{mailto:#1}{\nolinkurl{#1}}}
\renewcommand*\Affilfont{\normalfont\normalsize}
\newcommand\affilcr{\protect\\ \protect\Affilfont}
\makeatletter
  \renewcommand\AB@affilsepx{\protect\\[0.5em]}
\makeatother

\begin{document}

\author{Minh N. B\`ui}
\affil{%
  Universit\"at Graz
  \affilcr
  Institut f\"ur Mathematik und Wissenschaftliches Rechnen, NAWI Graz
  \affilcr
  Heinrichstra\ss{}e 36, 8010 Graz, Austria
  \affilcr
  \email{minh.bui@uni-graz.at}
}

\title{On a Lemma by Br\'ezis and Haraux}

\date{~}

\thispagestyle{plain.scrheadings}

\maketitle

\begin{abstract}
  We propose several applications of an often overlooked part of the
  1976 paper by Br\'ezis and Haraux, in which the Br\'ezis--Haraux
  theorem was established.
  Our results unify and extend various existing ones on the range of a
  linearly composite monotone operator and provide new insight into
  their seminal paper.
\end{abstract}

\begin{keywords}
  composite monotone operator;
  displacement operator;
  domain;
  range;
\end{keywords}

\begin{msc}
  47H04;
  47H05;
  49J53;
\end{msc}

In \cite[Th\'eor\`emes~3 et 4]{Brezis1976} Br\'ezis and Haraux
established under mild conditions on two set-valued monotone operators
$A$ and $B$ acting on a real Hilbert space the following
description for the range of their sum:
\begin{equation}
  \cran(A+B)=\overline{\ran A+\ran B}
  \quad\text{and}\quad
  \inte\ran(A+B)=\inte\brk!{\ran A+\ran B}.
\end{equation}
These two results are nowadays known as the \emph{Br\'ezis--Haraux
  theorem}, which is crucial in problems such as
establishing existence of solutions to nonlinear equations, describing
the range of monotone operators,
as well as deriving conditions for the maximal
monotonicity of a composite sum of monotone operators
\cite{Attouch1996,Jat2019,Livre1,Bauschke2016,Bauschke2020,Brezis1976,Chu1996,Simons2008}.
In some recent applications, this theorem turns out to be
insufficient and several Br\'ezis--Haraux-type results have therefore
been proposed
\cite{Combettes2023-svva,Combettes2022-siims,Jiang2023,Pennanen2001}.
In this short note, we present several applications of a powerful
--- but often overlooked --- lemma from the seminal work
\cite{Brezis1976} of Br\'ezis and Haraux (see \cref{l:1} below).
The proposed results unify and go beyond various existing ones.
In turn, we provide a new insight into \cite{Brezis1976}:
various results in the literature can, in fact, be derived from
\cref{l:1}, although the Br\'ezis--Haraux theorem does not seem to be
applicable in those scenarios.

Throughout,
\begin{equation}
  \text{$\HH$ is a real Hilbert space with scalar product
    $\scal{\Cdot}{\Cdot}$ and associated norm
    $\norm{\Cdot}$.}
\end{equation}
A set-valued operator $M\colon\HH\to 2^{\HH}$ is
called \emph{monotone} if
\begin{equation}
  \brk!{\forall (x,u)\in\gra M}
  \brk!{\forall (y,v)\in\gra M}\quad
  \scal{x-y}{u-v}\geq 0,
\end{equation}
where $\gra M=\menge{(x,u)\in\HH\times\HH}{u\in Mx}$ is the graph of
$M$. We call a monotone operator $M\colon\HH\to 2^{\HH}$
\emph{maximally monotone} if there exists no monotone operator whose
graph properly contains $\gra M$.
We shall mostly follow the notation of \cite{Livre1}, to which we refer
for background on monotone operator theory and for notions that are not
explicitly defined here.

It will be convenient to employ the following notation.

\begin{notation}
  \label{n:1}
  The interior of a subset $C$ of $\HH$ relative to the closed affine
  hull $\caff C$ of $C$, that is, the smallest closed affine subspace of
  $\HH$ containing $C$, is denoted by $\rint C$.
  We define a relation $\simeq$ on $2^{\HH}$ as follows:
  \begin{equation}
    \brk!{\forall C\in 2^{\HH}}
    \brk!{\forall D\in 2^{\HH}}\quad
    C\simeq D
    \quad\Longleftrightarrow\quad
    \begin{cases}
      \overline{C}=\overline{D}\\
      \rint C=\rint D.
    \end{cases}
  \end{equation}
\end{notation}

Additionally, we shall require the following property of the
relative interior operator $\rint$.
The transitivity of the relation $\simeq$ justifies writing
$C\simeq D\simeq E$ to designate $C\simeq D$ and $D\simeq E$,
where $C$, $D$, and $E$ are subsets of $\HH$.

\begin{lemma}
  \label{l:2}
  Let $C$ and $D$ be nonempty subsets of $\HH$ such that
  $\conv D\subset\overline{C}$ and $\rint(\conv D)\subset C\subset D$.
  Then $C\simeq D\simeq\conv D$.
\end{lemma}
\begin{proof}
  Clearly $\overline{C}=\overline{D}=\cconv D$.
  Therefore, in view of\footnote{In the second equality,
    it is crucial to take the closed affine
    hull operator and not the affine hull.
    Indeed, if $\HH$ is infinite-dimensional,
    then for every proper vector subspace $E$ of $\HH$ that is dense in
    $\HH$, $\aff E=E\neq\HH=\aff(\overline{E})$.}
  \begin{equation}
    \label{e:ye21}
    \brk!{\forall E\in 2^{\HH}\setminus\set{\emp}}\quad
    \caff(\conv E)=\caff E=\caff(\overline{E}),
  \end{equation}
  the sets $C$, $D$, and $\conv D$ have the same closed affine hull.
  In turn, we work in this common closed affine hull to deduce that
  $\rint(\conv D)
  \subset\rint C
  \subset\rint D
  \subset\rint(\conv D)$, which confirms that
  $\rint C=\rint D=\rint(\conv D)$.
\end{proof}

\begin{remark}
  If $\HH$ is finite-dimensional, then
  \begin{equation}
    \label{e:43}
    \brk!{\forall C\in 2^{\HH}}
    \brk!{\forall D\in 2^{\HH}}
    \quad
    \rint(\conv D)\subset C\subset D
    \quad\Longrightarrow\quad
    C\simeq D\simeq\conv D
  \end{equation}
  due to the fact that $\brk{\forall D\in 2^{\HH}}$
  $\overline{\rint(\conv D)}=\cconv D$.
  However, this argument breaks down when $\HH$ is
  infinite-dimensional since it is possible that
  $D\neq\emp$ and $\rint(\conv D)=\emp$.
  To see this, suppose that $\HH$ is infinite-dimensional,
  let $D$ be a proper vector subspace of $\HH$ that is dense in $\HH$,
  let $u\in D\setminus\set{0}$, and set $C=\RR u$.
  Then $\rint(\conv D)=\inte D=\emp\subset C\subset D$
  and $\overline{C}=C\neq\HH=\overline{D}$.
  This necessitates the additional assumption that
  $\conv D\subset\overline{C}$ in \cref{l:2}.
\end{remark}

We are now in a position to state the aforementioned lemma by Br\'ezis
and Haraux.

\begin{lemma}[Br\'ezis--Haraux]
  \label{l:1}
  Let $M\colon\HH\to 2^{\HH}$ be a maximally monotone operator
  and let $W$ be a subset of $\HH$ such that
  \begin{equation}
    \label{e:bh}
    (\forall w\in W)(\exi z\in\HH)\quad
    \inf_{(x,u)\in\gra M}\scal{x-z}{u-w}>\minf.
  \end{equation}
  Then the following hold:
  \begin{enumerate}
    \item\label{l:1i}\textup{\cite[Lemme~1]{Brezis1976}}
      $\conv W\subset\cran M$ and $\inte\conv W\subset\ran M$.
    \item\label{l:1ii}\textup{\cite[Lemme~1']{Brezis1976}}
      Suppose that $\ran M\subset\overline{W}$.
      Then $\rint(\conv W)\subset\ran M$.
  \end{enumerate}
\end{lemma}

The following direct consequence of \cref{l:1} will be crucial in
establishing our results.

\begin{corollary}
  \label{c:2}
  Let $M\colon\HH\to 2^{\HH}$ be a maximally monotone operator
  and let $W$ be a subset of $\HH$ such that
  \begin{equation}
    \label{e:bh2}
    (\forall w\in W)(\exi z\in\HH)\quad
    \inf_{(x,u)\in\gra M}\scal{x-z}{u-w}>\minf.
  \end{equation}
  Suppose that $\ran M\subset W$.
  Then $\ran M\simeq W\simeq\conv W$.
\end{corollary}

\begin{remark}
  \leavevmode
  \begin{enumerate}
    \item
      \cref{l:1} remains valid in reflexive Banach spaces;
      see \cite[Lemma~2.1]{Reich1979} and
      \cite[Proposition~2.6]{Chu1996}.
    \item
      For variants and generalizations of \cref{l:1},
      see \cite[Remarque~5]{Brezis1976}.
    \item
      In \cref{c:2}, taking $W=\ran M$ yields
      \begin{equation}
        \label{e:ri}
        \rint(\ran M)=\rint\brk!{\conv(\ran M)},
      \end{equation}
      which is the relativized version of
      the identity
      $\inte(\ran M)=\inte(\conv(\ran M))$
      (see, e.g., \cite[Lemma~31.1]{Simons2008}).
      In the finite-dimensional setting,
      \cref{e:ri} goes back to the work \cite{Minty1961} of Minty and,
      thanks to \cite[Fact~6.14(i)]{Livre1}, we always have
      $\rint(\ran M)\neq\emp$.
      However, when $\HH$ is infinite-dimensional,
      \cite[Example~20.41]{Livre1} provides an example of a maximally
      monotone operator $M\colon\HH\to\HH$ whose range is a proper
      vector subspace of $\HH$ which is dense in $\HH$; in particular,
      we have $\rint(\ran M)=\inte(\ran M)=\emp$.
      In passing, we note that \cref{e:ri} can also be deduced by
      combining \cite[Corollary~3.7]{Penot2005} and
      \cite[Proposition~3.1.5]{Zalinescu2002}.
  \end{enumerate}
\end{remark}

Our first illustration of \cref{l:1} is a description of the domain of
a maximally monotone operator; see \cite[Corollary~3.7]{Penot2005} 
and \cite[Theorem~31.2]{Simons2008}.

\begin{example}
  \label{ex:9}
  Let $M\colon\HH\to 2^{\HH}$ be maximally monotone. Then
  \begin{equation}
    \label{e:4}
    \dom M\simeq
    \Menge3{z\in\HH}{(\exi w\in\HH)\,\,\inf_{(x,u)\in\gra
        M}\scal{x-z}{u-w}>\minf}.
  \end{equation}
\end{example}
\begin{proof}
  Apply \cref{c:2} to $M^{-1}$ and the set on the right-hand side of
  \cref{e:4}, and observe that it contains $\dom M=\ran M^{-1}$ due to
  the monotonicity of $M$.
\end{proof}

Next, we provide a formula for the range of a type of linearly composed
operator, which will be shown to bring together
\cite[Th\'eor\`emes~3 et 4]{Brezis1976}.
We first recall some terminologies.
A monotone operator
$M\colon\HH\to 2^{\HH}$ is said to be \emph{$3^*$-monotone} ---
in the sense of Br\'ezis and Haraux, who call this ``propri\'et\'e
$(\ast)$'' in their seminal work \cite{Brezis1976}
--- if
\begin{equation}
  (\forall x\in\dom M)(\forall u\in\ran M)\quad
  \inf_{(y,v)\in\gra M}\scal{x-y}{u-v}>\minf.
\end{equation}
Now let $\GG$ be a real Hilbert space. A set-valued operator
$L\colon\HH\to 2^{\GG}$ is said to be \emph{linear}
if its graph $\gra L=\menge{(x,u)\in\HH\times\GG}{u\in Lx}$
is a vector subspace of $\HH\times\GG$ \cite{Arens1961}.
The \emph{adjoint} of a linear operator
$L\colon\HH\to 2^{\GG}$ is the linear operator
$L^*\colon\GG\to 2^{\HH}$ whose graph is
\begin{equation}
  \label{e:adj}
  \gra L^*
  =\menge{(v,u)\in\GG\times\HH}{\brk!{\forall (x,y)\in\gra L}\,\,
    \scal{x}{u}=\scal{y}{v}}.
\end{equation}

\begin{theorem}
  \label{t:1}
  Let $(\GG_k)_{k\in K}$ be a finite family of real Hilbert spaces.
  Let $A\colon\HH\to 2^{\HH}$ be monotone and,
  for every $k\in K$, let $B_k\colon\GG_k\to 2^{\GG_k}$ be
  $3^*$-monotone and let $L_k\colon\HH\to 2^{\GG_k}$ be linear.
  Set
  \begin{equation}
    \label{e:D}
    D=\bigcap_{k\in K}L_k^{-1}(\dom B_k).
  \end{equation}
  Suppose that $A+\sum_{k\in K}L_k^*\circ B_k\circ L_k$ is maximally
  monotone. Then $D\neq\emp$ and
  \begin{equation}
    \ran\brk3{A+\sum_{k\in K}L_k^*\circ B_k\circ L_k}
    \simeq A(D)+\sum_{k\in K}L_k^*(\ran B_k).
  \end{equation}
\end{theorem}
\begin{proof}
  We apply \cref{c:2} to
  \begin{equation}
    M=A+\sum_{k\in K}L_k^*\circ B_k\circ L_k
    \quad\text{and}\quad
    W=A(D)+\sum_{k\in K}L_k^*(\ran B_k).
  \end{equation}
  Observe that
  \begin{align}
    \emp
    &\neq\dom M
    \nonumber\\
    &=(\dom A)\cap\dom\brk3{\sum_{k\in K}L_k^*\circ B_k\circ L_k}
    \nonumber\\
    &=(\dom A)\cap\brk3{\bigcap_{k\in K}L_k^{-1}\brk!{B_k^{-1}(\dom L_k^*)}}
    \nonumber\\
    &\subset(\dom A)\cap D
  \end{align}
  and therefore $D\neq\emp$ and $\ran M\subset W$.
  Thus, it remains to verify \cref{e:bh2}.
  Towards this goal, take $w\in W$, say
  \begin{equation}
  w=p+\sum_{k\in K}q_k,
  \quad\text{where}\,\,
  \begin{cases}
    p\in Az\,\,\text{for some $z\in D$}\\
    \text{for every $k\in K$},\,\,
    q_k\in L_k^*s_k\,\,\text{for some $s_k\in\ran B_k$}.
  \end{cases}
  \end{equation}
  Since $z\in D$, it results from \cref{e:D} that,
  for every $k\in K$, there exists $d_k\in\dom B_k$ for which
  $d_k\in L_kz$.
  Now let $(x,u)\in\gra M$, say
  \begin{multline}
    u=r+\sum_{k\in K}u_k,\\
    \text{where}\,\,
    \begin{cases}
      r\in Ax\\
      \text{for every $k\in K$},\,\,
      u_k\in L_k^*v_k\,\,\text{for some $(y_k,v_k)\in\gra B_k$
        with $y_k\in L_kx$}.
    \end{cases}
  \end{multline}
  By linearity,
  \begin{equation}
    \brk{\forall k\in K}\quad
    y_k-d_k\in L_k(x-z)
    \quad\text{and}\quad
    u_k-q_k\in L_k^*(v_k-s_k).
  \end{equation}
  Thus, we derive from the monotonicity of $A$ and \cref{e:adj} that
  \begin{align}
    \label{e:11}
    \scal{x-z}{u-w}
    &=\scal{x-z}{r-p}+\sum_{k\in K}\scal{x-z}{u_k-q_k}
    \nonumber\\
    &\geq\sum_{k\in K}\scal{y_k-d_k}{v_k-s_k}
    \nonumber\\
    &\geq\sum_{k\in K}\inf_{(b_k,f_k)\in\gra B_k}\scal{b_k-d_k}{f_k-s_k}.
  \end{align}
  Hence, taking the infimum over $(x,u)\in\gra M$ and appealing to the
  $3^*$-monotonicity of the operators $(B_k)_{k\in K}$, we obtain
  \begin{equation}
    \inf_{(x,u)\in\gra M}\scal{x-z}{u-w}
    \geq\sum_{k\in K}\inf_{(b_k,f_k)\in\gra B_k}\scal{b_k-d_k}{f_k-s_k}
    >\minf,
  \end{equation}
  which concludes the proof.
\end{proof}

\begin{remark}
  \label{r:2}
  Here are several observations on \cref{t:1}.
  \begin{enumerate}
    \item
      Letting $m=2$, $\GG_1=\GG_2=\HH$, $L_1=L_2=\Id$, and $A=0$,
      we recover \cite[Th\'eor\`eme~3]{Brezis1976}
      and its finite-dimensional version
      \cite[Theorem~3.13(i)]{Bauschke2013}.
    \item
      Letting $m=1$, $\GG_1=\HH$, and $L_1=\Id$, as well as assuming, in
      addition, that $\dom A\subset\dom B_1$,
      we recover \cite[Th\'eor\`eme~4]{Brezis1976}
      and its finite-dimensional version
      \cite[Theorem~3.13(ii)]{Bauschke2013}.
    \item
      Suppose additionally that, for every $k\in K$,
      $L_k\colon\HH\to\GG_k$ is bounded.
      Then letting $A=0$ yields at once
      \cite[Lemma~2.5]{Combettes2022-siims}.
      Further specialization yields
      \cite[Theorem~4.5(i)]{Bauschke2013},
      \cite[Lemma~3.1(ii)]{Bauschke2018}, and
      \cite[Theorem~5.1]{Bauschke2020}.
    \item
      \cref{t:1} remains valid in reflexive real Banach spaces.
      Indeed, let $\XX$ be a reflexive real Banach space
      with topological dual $\XX^*$, and let
      $A\colon\XX\to 2^{\XX^*}$ be monotone.
      Further, let $(\YY_k)_{k\in K}$ be a finite family of reflexive
      real Banach spaces and, for every $k\in K$, let
      $B_k\colon\YY_k\to 2^{\YY_k^*}$ be $3^*$-monotone
      in the sense that it is monotone and satisfies
      \begin{equation}
        (\forall y\in\dom B_k)(\forall y^*\in\ran B_k)\quad
        \inf_{(z,z^*)\in\gra B_k}\pair{y-z}{y^*-z^*}>\minf,
      \end{equation}
      and let $L_k\colon\XX\to 2^{\YY_k}$ be linear
      in the sense that
      $\gra L_k=\menge{(x,y)\in\XX\times\YY_k}{y\in L_kx}$
      is a vector subspace of $\XX\times\YY_k$.
      Suppose that $A+\sum_{k\in K}L_k^*\circ B_k\circ L_k$ is maximally
      monotone, where for every $k\in K$, the linear operator
      $L_k^*\colon\YY_k^*\to 2^{\XX^*}$ is defined via its graph
      \begin{equation}
        \gra L_k^*=\menge{(y^*,x^*)\in\YY_k^*\times\XX^*}{
          \brk!{\forall (x,y)\in\gra L_k}\,\,
          \pair{x}{x^*}=\pair{y}{y^*}}.
      \end{equation}
      (Here $\pair{\Cdot}{\Cdot}$ stands for the canonical pairing
      between a Banach space and its topological dual.)
      Then
      \begin{equation}
        \ran\brk3{A+\sum_{k\in K}L_k^*\circ B_k\circ L_k}
        \simeq A(D)+\sum_{k\in K}L_k^*(\ran B_k),
      \end{equation}
      where $D=\bigcap_{k\in K}L_k^{-1}(\dom B_k)$.
      To prove this, argue as in the proof of \cref{t:1}
      and use \cite[Lemma~2.1]{Reich1979} and
      \cite[Proposition~2.6]{Chu1996} instead of \cref{c:2}.
      This result subsumes, in particular,
      \cite[Theorem~5]{Pennanen2001}.
    \item
      Sufficient conditions for the maximal monotonicity
      of $A+\sum_{k\in K}L_k^*\circ B_k\circ L_k$ can be derived from
      the results of \cite{Pennanen2000}.
  \end{enumerate}
\end{remark}

\begin{example}
  \label{ex:3}
  We utilize the example given in \cite[Remarque~8]{Brezis1976}
  to provide a situation where the assumption of \cref{t:1} is
  fulfilled, but the conclusions of
  \cite[Th\'eor\`emes~3 et 4]{Brezis1976} fail.
  Consider the $(\pi/2)$-rotation operator
  \begin{equation}
    A\colon\RR^2\to\RR^2\colon
    (\xi_1,\xi_2)\mapsto
    (-\xi_2,\xi_1)
  \end{equation}
  and the normal cone operator to $\RR\times\set{0}$
  \begin{equation}
    B=N_{\RR\times\set{0}}\colon\RR^2\to 2^{\RR^2}\colon
    (\xi_1,\xi_2)\mapsto
    \begin{cases}
      \set{0}\times\RR,&\text{if}\,\,\xi_2=0;\\
      \emp,&\text{otherwise}.
    \end{cases}
  \end{equation}
  Then
  \begin{equation}
    \ran(A+B)
    =\set{0}\times\RR
    \neq\RR^2
    =\ran A+\ran B.
  \end{equation}
  On the other hand, $A(\dom B)=A(\RR\times\set{0})=\set{0}\times\RR$
  and thus
  \begin{equation}
    A(\dom B)+\ran B=\set{0}\times\RR=\ran(A+B).
  \end{equation}
\end{example}

\begin{corollary}
  \label{c:1}
  Let $A\colon\HH\to 2^{\HH}$ be monotone and
  let $B\colon\HH\to 2^{\HH}$ be $3^*$-monotone and surjective.
  Suppose that $A+B$ is maximally monotone.
  Then $A+B$ is surjective.
\end{corollary}

\begin{example}[range of Kuhn--Tucker operator]
  \label{ex:1}
  Let $\GG$ be a real Hilbert space,
  let $A\colon\HH\to 2^{\HH}$ and $B\colon\GG\to 2^{\GG}$ be
  maximally monotone,
  let $L\colon\HH\to\GG$ be linear and bounded,
  and let
  \begin{equation}
    \kt\colon\HH\times\GG\to 2^{\HH\times\GG}\colon
    (x,v)\mapsto (Ax+L^*v)\times(-Lx+B^{-1}v)
  \end{equation}
  be the \emph{Kuhn--Tucker operator} associated with
  the problem of finding a zero of
  $A+L^*\circ B\circ L$ \cite{Combettes2024-acnu}.
  Then the following hold:
  \begin{enumerate}
    \item\label{ex:1i}
      Suppose that $B$ is $3^*$-monotone and set
      $T\colon\HH\times\HH\to\HH\times\GG\colon (x,u)\mapsto(u,-Lx)$.
      Then
      \begin{equation}
        \ran\kt\simeq
        T\brk{\gra A}+\brk!{L^*(\ran B)}\times\dom B.
      \end{equation}
    \item\label{ex:1ii}
      Suppose that $A$ and $B$ are $3^*$-monotone. Then
      \begin{equation}
        \label{e:6}
        \ran\kt\simeq
        \brk!{\ran A+L^*(\ran B)}\times\brk!{-L(\dom A)+\dom B}.
      \end{equation}
  \end{enumerate}
\end{example}
\begin{proof}
  We denote by $\HHH$ the Hilbert direct sum of $\HH$ and $\GG$.
  On account of \cite[Proposition~26.32(iii)]{Livre1}, $\kt$ is
  maximally monotone.

  \cref{ex:1i}:
  We deduce from \cref{t:1} applied to the
  monotone operator
  $\boldsymbol{A}\colon\HHH\to 2^{\HHH}\colon
  (x,v)\mapsto\brk{Ax+L^*v}\times\set{-Lx}$
  and the $3^*$-monotone operator
  $\boldsymbol{B}\colon\HHH\to 2^{\HHH}\colon(x,v)\mapsto
  \set{0}\times B^{-1}v$ that
  \begin{align}
    \ran\kt
    &=\ran(\boldsymbol{A}+\boldsymbol{B})
    \nonumber\\
    &\simeq
    \boldsymbol{A}\brk!{(\dom\boldsymbol{A})\cap(\dom\boldsymbol{B})}
    +\set{0}\times\dom B
    \nonumber\\
    &=\boldsymbol{A}\brk!{(\dom A)\times(\ran B)}
    +\set{0}\times\dom B
    \nonumber\\
    &=\menge{(u+L^*v,-Lx)}{(x,u)\in\gra A\,\,\text{and}\,\,v\in\ran B}
    +\set{0}\times\dom B
    \nonumber\\
    &=\menge{(u,-Lx)}{(x,u)\in\gra A}+\brk!{L^*(\ran B)}\times\dom B
    \nonumber\\
    &=T\brk{\gra A}+\brk!{L^*(\ran B)}\times\dom B,
  \end{align}
  as claimed.

  \cref{ex:1ii}:
  Apply \cref{t:1} to the monotone operator
  $\boldsymbol{A}\colon\HHH\to\HHH\colon
  (x,v)\mapsto(L^*v,-Lx)$ and the $3^*$-monotone operator
  $\boldsymbol{B}\colon\HHH\to 2^{\HHH}\colon
  (x,v)\mapsto Ax\times B^{-1}v$.
\end{proof}

\begin{remark}
  The part
  \begin{equation}
    \cran\kt
    =\overline{\brk!{\ran A+L^*(\ran B)}\times\brk!{-L(\dom A)+\dom B}}
  \end{equation}
  in the conclusion of \cref{ex:1}\cref{ex:1ii}
  is a main result of \cite{Jiang2023}.
\end{remark}

Our next application of \cref{l:1} concerns the range of the so-called
displacement  operator, which has been the underlying topic of several
recent papers; see, e.g.,
\cite{Bauschke2016,Bauschke2020,Moursi2022} and the references therein.
Let us preface this with some background.
Given a maximally monotone operator $M\colon\HH\to 2^{\HH}$,
the \emph{resolvent} of $M$ is the single-valued operator
\begin{equation}
  J_M=(\Id+M)^{-1}
\end{equation}
and we have $\ran J_M=\dom M$, $\ran J_{M^{-1}}=\ran M$, and
\begin{equation}
  \label{e:7}
  J_{M^{-1}}=\Id-J_M
  \quad\text{and}\quad
  (\forall x\in\HH)\;\;\brk!{J_Mx,J_{M^{-1}}x}\in\gra M;
\end{equation}
additionally, the \emph{reflected resolvent} of $M$ is
\begin{equation}
  R_A=2J_A-\Id.
\end{equation}

\begin{theorem}
  \label{t:2}
  Let $A\colon\HH\to 2^{\HH}$ and $B\colon\HH\to 2^{\HH}$ be maximally
  monotone, and denote by $T$ the \emph{Douglas--Rachford operator}
  associated with $(A,B)$, that is,
  \begin{equation}
    T=\Id-J_A+J_B\circ R_A.
  \end{equation}
  Suppose that $W$ is a subset of $\HH$
  which contains $\ran(\Id-T)$ and satisfies
  \begin{equation}
    \label{e:2}
    (\forall w\in W)\brk!{\exi(a,b,p,q)\in\HH^4}\quad
    \begin{cases}
      w=a-b=p+q\\
      \displaystyle
      \inf_{(x,u)\in\gra A}\scal{x-a}{u-p}>\minf\\
      \displaystyle
      \inf_{(y,v)\in\gra B}\scal{y-b}{v-q}>\minf.
    \end{cases}
  \end{equation}
  Then $\ran(\Id-T)\simeq W\simeq\conv W$.
\end{theorem}
\begin{proof}
  It results from \cite[Proposition~26.1(iii)(a) and
  Example~20.29]{Livre1} that $M=\Id-T$
  is maximally monotone.
  Now take $w\in W$ and let $(a,b,p,q)\in\HH^4$ be such that
  $w=a-b=p+q$, and that
  \begin{equation}
    \label{e:8}
      \inf_{(x,u)\in\gra A}\scal{x-a}{u-p}>\minf
      \quad\text{and}\quad
      \inf_{(y,v)\in\gra B}\scal{y-b}{v-q}>\minf.
  \end{equation}
  For every $x\in\HH$, applying the identity
  \begin{equation}
    (\forall r\in\HH)(\forall s\in\HH)(\forall t\in\HH)\quad
    \scal{r}{s-t}=\norm{s-t}^2+\scal{s}{r-s}+\scal{t}{2s-t-r}
  \end{equation}
  to $r=x-(a+p)$, $s=J_Ax-a$, and $t=J_B(R_Ax)-b$, as well as observing
  that $s-t=x-Tx-a+b=Mx-w$, we derive from the resolvent identity
  in \cref{e:7} that
  \begin{align}
    &\scal{x-(a+p)}{Mx-w}
    \nonumber\\
    &\hskip 8mm
    =\norm{Mx-w}^2
    +\scal{J_Ax-a}{J_{A^{-1}}x-p}
    +\scal{J_B(R_Ax)-b}{J_{B^{-1}}(R_Ax)-q}
    \nonumber\\
    &\hskip 8mm
    \geq\scal{J_Ax-a}{J_{A^{-1}}x-p}
    +\scal{J_B(R_Ax)-b}{J_{B^{-1}}(R_Ax)-q}.
  \end{align}
  Hence, using \cref{e:7,e:8}, we obtain
  \begin{align}
    &\inf_{x\in\HH}\scal{x-(a+p)}{Mx-w}
    \nonumber\\
    &\hskip 8mm
    \geq\inf_{x\in\HH}\scal{J_Ax-a}{J_{A^{-1}}x-p}
    +\inf_{x\in\HH}\scal{J_B(R_Ax)-b}{J_{B^{-1}}(R_Ax)-q}
    \nonumber\\
    &\hskip 8mm
    \geq\inf_{(x,u)\in\gra A}\scal{x-a}{u-p}
    +\inf_{(y,v)\in\gra B}\scal{y-b}{v-q}
    \nonumber\\
    &\hskip 8mm
    >\minf.
  \end{align}
  Consequently, \cref{c:2} yields the conclusion.
\end{proof}

\begin{remark}
  \label{r:1}
  We recover at once various results on the range of the
  Douglas--Rachford operator by choosing suitable instances of $W$ in
  \cref{t:2}. In fact, consider the operators $A$, $B$, and $T$ as in
  \cref{t:2}. Then, since
  \begin{equation}
   \Id-T=J_A-J_B\circ R_A=J_{A^{-1}}+J_{B^{-1}}\circ R_A,
  \end{equation}
  we get
  \begin{equation}
    \label{e:3}
    \ran(\Id-T)\subset(\dom A-\dom B)\cap(\ran A+\ran B).
  \end{equation}
  \begin{enumerate}
    \item\label{r:1i}
      Suppose that $A$ and $B$ are $3^*$-monotone.
      Then clearly
      \begin{equation}
        W=(\dom A-\dom B)\cap(\ran A+\ran B)
      \end{equation}
      satisfies \cref{e:2}.
      Thus, in view of \cref{e:3}, we deduce from \cref{t:2} that
      \begin{equation}
        \ran(\Id-T)
        \simeq(\dom A-\dom B)\cap(\ran A+\ran B),
      \end{equation}
      and we therefore recover \cite[Theorem~2.5(i)]{Moursi2022}
      and \cite[Theorem~5.2(ii)]{Bauschke2016}.
    \item\label{r:1ii}
      Suppose that $A$ is $3^*$-monotone with $\dom A=\HH$.
      Then
      \begin{equation}
        W=\ran A+\ran B
      \end{equation}
      satisfies \cref{e:2}. In fact, let $w\in W$, say
      $w=p+q$, where $p\in\ran A$ and $q\in\ran B$.
      In turn, let $b\in\HH$ be such that $(b,q)\in\gra B$,
      and set $a=b+w$.
      Then $a\in\HH=\dom A$ and the $3^*$-monotonicity of $A$ yields
      $\inf_{(x,u)\in\gra A}\scal{x-a}{u-p}>\minf,$ while the
      monotonicity of $B$ ensures that
      $\inf_{(y,v)\in\gra B}\scal{y-b}{v-q}=0$.
      Hence, it follows from \cref{e:3} and \cref{t:2} that
      \begin{equation}
        \ran(\Id-T)\simeq\ran A+\ran B,
      \end{equation}
      which shows that \cite[Theorem~2.5(ii)]{Moursi2022} is a special
      case of \cref{t:2}. Nevertheless, only the identity
      $\cran(\Id-T)=\overline{\ran A+\ran B}$ is established in
      \cite[Theorem~2.5(ii)]{Moursi2022}.
    \item Along the same lines of \cref{r:1ii},
      one can verify that, if
      $A$ is $3^*$-monotone with $\ran A=\HH$, then $W=\dom A-\dom B$
      satisfies \cref{e:2}, and \cref{t:2} therefore yields
      \cite[Theorem~2.5(iii)]{Moursi2022},
      where only the identity
      $\cran(\Id-T)=\overline{\dom A-\dom B}$ is established.
    \item Suppose that there exist proper lower semicontinuous convex
      functions $\varphi\colon\HH\to\RX$ and
      $\psi\colon\HH\to\RX$ such that
      $A=\partial\varphi$ and $B=\partial\psi$.
      Then
      \begin{equation}
        W=(\dom\varphi-\dom\psi)\cap(\dom\varphi^*+\dom\psi^*)
      \end{equation}
      satisfies \cref{e:2}. Indeed, this follows from the fact that,
      given a proper lower semicontinuous convex function
      $\theta\colon\HH\to\RX$, we have \cite[Example~20.55]{Livre1}
      \begin{equation}
        (\forall y\in\dom\theta)(\forall v\in\dom\theta^*)
        \quad
        \inf_{(x,u)\in\gra\partial\theta}\scal{x-y}{u-v}>\minf.
      \end{equation}
      \cref{t:2} therefore subsumes
      \cite[Theorem~3.3(i)]{Moursi2022}
      and \cite[Corollary~6.5(i)]{Bauschke2016}.
  \end{enumerate}
\end{remark}

We end this note with an illustration of \cref{t:2} which extends
\cite[Theorem~3.2]{Bauschke2020}.

\begin{example}
  \label{ex:2}
  Let $A\colon\HH\to 2^{\HH}$ be $3^*$-monotone and maximally monotone
  with $\dom A=\HH$, and let and $B\colon\HH\to 2^{\HH}$ be maximally
  monotone. Then
  \begin{equation}
    \label{e:5}
    \ran(\Id-R_B\circ R_A)\simeq\ran(\Id-R_A)+\ran(\Id-R_B).
  \end{equation}
\end{example}
\begin{proof}
  We deduce from \cref{t:2} and \cref{r:1}\cref{r:1ii} that
  \begin{equation}
    \ran(J_A-J_B\circ R_A)
    =\ran(\Id-T)
    \simeq\ran A+\ran B.
  \end{equation}
  On the other hand, a simple algebraic manipulation gives
  \begin{equation}
    \begin{cases}
    \Id-R_B\circ R_A=2(J_A-J_B\circ R_A)\\
    \ran(\Id-R_A)=2\ran(\Id-J_A)=2\ran J_{A^{-1}}=2\ran A\\
    \ran(\Id-R_B)=2\ran(\Id-J_B)=2\ran J_{B^{-1}}=2\ran B.
  \end{cases}
  \end{equation}
  Altogether, \cref{e:5} follows.
\end{proof}

\begin{remark}
  Consider the setting of \cref{ex:2}.
  In \cite[Theorem~3.2]{Bauschke2020} the operator $A$ is further
  required to be cocoercive, and only the identity
  \begin{equation}
    \cran(\Id-R_B\circ R_A)=\overline{\ran(\Id-R_A)+\ran(\Id-R_B)}
  \end{equation}
  is established.
\end{remark}

\newpage
\printbibliography

\end{document}